\documentclass[12pt]{article}

\usepackage{amsmath}
\usepackage{amssymb}
\usepackage{amscd}
\usepackage{amsthm}
\theoremstyle{plain}
\newtheorem{Thm}{Theorem}[section]

\newtheorem{Prop}[Thm]{Proposition}
\newtheorem{Conj}[Thm]{Conjecture}

\theoremstyle{definition}

\theoremstyle{remark}
\newtheorem{Rem}[Thm]{Remark}
\numberwithin{equation}{section}

\begin{document}
\title{An elementary semi-ampleness result \\ for log canonical divisors}
\author{\thanks{2010 \textit{Mathematics Subject Classification}: 14E30}
\thanks{\textit{Key words and phrases}: log canonical divisor, semi-ample} Shigetaka FUKUDA}

\date{\empty}

\maketitle \thispagestyle{empty}
\pagestyle{myheadings}
\markboth{Shigetaka FUKUDA}{Log canonical divisors}
\begin{abstract}
If the log canonical divisor on a projective variety with only Kawamata log terminal singularities is numerically equivalent to some semi-ample $\mathbf{Q}$-divisor, then it is semi-ample.
\end{abstract}

In this note, every algebraic variety is defined over the field $\mathbf{C}$ of complex numbers.
We follow the terminology and notation in \cite{Utah}.

\begin{Thm}[Main Theorem]
Let $(X,\Delta)$ be a projective variety with only Kawamata log terminal singularities.
Assume that the log canonical divisor $K_X + \Delta$ is numerically equivalent to some semi-ample $\mathbf{Q}$-Cartier $\mathbf{Q}$-divisor.
Then $K_X+\Delta$ is semi-ample.
\end{Thm}

\begin{Rem}
Divisors that are numerically equivalent to semi-ample $\mathbf{Q}$-divisors are nef.
So Main Theorem is a corollary of the famous log abundance conjecture for Kawamata log terminal pairs.
\end{Rem}

\begin{Rem}
After the earlier draft of the manuscript was written out, Campana-Koziarz-Paun (\cite{CKP}) showed that Main Theorem holds under the weaker condition that $K_X + \Delta$ is numerically equivalent to some nef and abundant $\mathbf{Q}$-Cartier $\mathbf{Q}$-divisor.
\end{Rem}

For proof we cite the following two results.
The first is the $\mathbf{Q}$-linear triviality (Proposition \ref{Prop:KaN}) due to Kawamata and Nakayama and the second is the relative semi-ampleness (Proposition \ref{Prop:KNF}) due to Kawamata, Nakayama and Fujino.

\begin{Prop}[\cite{Ka2} Theorem 8.2, \cite{Na04} Corollary V.4.9]\label{Prop:KaN}
Let $(X,\Delta)$ be a projective variety with only Kawamata log terminal singularities.
Assume that the log canonical divisor $K_X + \Delta$ is numerically trivial.
Then $K_X+\Delta$ is $\mathbf{Q}$-linearly trivial.
\end{Prop}

\begin{Rem}
Ambro (\cite{Am} Theorem 0.1) gives an alternative proof to Proposition \ref{Prop:KaN}, by providing some log canonical bundle formula and applying it to the Albanese morphism.
\end{Rem}

\begin{Rem}
In the statement of Main Theorem, Proposition \ref{Prop:KaN} gives the special case where $K_X + \Delta$ is numerically equivalent to the trivial divisor 0, which is, of course, semi-ample.
\end{Rem}

\begin{Prop}[\cite{Ka1} Theorem 6.1, \cite{Na} Theorem 5, \cite{Fu} Theorem 1.1]\label{Prop:KNF}
Let $(X, \Delta)$ be a projective variety with only Kawamata log terminal singularities and $f: X \to Y$ a surjective morphism of normal projective varieties with only connected fibers.
If $K_X + \Delta$ is $f$-nef and $(K_X + \Delta) \vert_F$ is semi-ample for a general fiber $F$ of $f$, then the log canonical divisor $K_X + \Delta$ is $f$-semi-ample.
\end{Prop}

\begin{proof}[Proof of Main Theorem]
Let $D$ be a semi-ample $\mathbf{Q}$-Cartier $\mathbf{Q}$-divisor that is numerically equivalent to $K_X+\Delta$.
We consider the surjective morphism $f:X \to Y$ of normal projective varieties with only connected fibers, defined by the linear space $H^0(X,{\cal O}_X (lD))$ for a sufficiently large and divisible integer $l$.
Then $lD = f^* A$ for some ample divisor $A$ on $Y$.

The log canonical divisor $K_X+\Delta$ is $f$-nef.
Furthermore the pair $(F, \Delta \vert_F)$ is Kawamata log terminal and, from a Kawamata-Nakayama result (Proposition \ref{Prop:KaN}), the log canonical divisor $K_F + (\Delta \vert_F) = (K_X + \Delta) \vert_F$ is $\mathbf{Q}$-linearly trivial for a general fiber $F$ of $f$.

Thus a relative semi-ampleness result due to Kawamata-Nakayama-Fujino (Proposition \ref{Prop:KNF}) gives the surjective morphism $g:X \to Z$ of normal projective varieties with only connected fibers, defined by the sheaf $f_* {\cal O}_X (m(K_X + \Delta))$ for a sufficiently large and divisible integer $m$, with the structure morphism $h:Z \to Y$ such that $hg=f$.
Then $m(K_X + \Delta) = g^* B$ for some $h$-ample divisor $B$ on $Z$.

For a curve C on X, if f(C) is a point then also g(C) is a point, because $0 = m(f^* A,C) = m (l(K_X + \Delta),C) = l(m(K_X + \Delta),C) = l(g^* B,C)$.
Thus the morphism $h$ is birational and finite.
This means that h is the identity morphism by virtue of Zariski's Main Theorem.

Hence the divisors $mA$ and $lB$ are numerically equivalent to each other on $Y$, because $f^* (mA-lB)$ is numerically trivial on $X$.
Thus $lB$ is ample, from the fact that $mA$ is ample.
Consequently $K_X + \Delta$ is semi-ample.
\end{proof}

Finally, by relaxing the condition concerning singularities, we propose the following subconjecture towards the famous log abundance conjecture.

\begin{Conj}\label{Conj:LC}
Let $(X,\Delta)$ be a projective variety with only log canonical singularities.
Assume that the log canonical divisor $K_X + \Delta$ is numerically equivalent to some semi-ample $\mathbf{Q}$-Cartier $\mathbf{Q}$-divisor.
Then $K_X+\Delta$ is semi-ample.
\end{Conj}

\begin{Rem}
Kawamata's result (\cite{Ka3}) proves Conjecture \ref{Conj:LC} in the case where $\Delta$ is a reduced simple normal crossing divisor on a smooth variety $X$ and where $K_X + \Delta$ is numerically trivial.
\end{Rem}

\begin{Rem}
Recently Gongyo (\cite{Go}) proved Conjecture \ref{Conj:LC} in dimension $\leq 4$.
Moreover he extended Kawamata's result for the numerically trivial log canonical divisors $K_X + \Delta$ to the case of projective semi-log canonical pairs $(X, \Delta)$.
His proof depends on Proposition \ref{Prop:KaN}, the minimal model program (\cite{BCHM}) with scaling and the theory of semi-log canonical pairs (\cite{Fu00}) due to Fujino.
\end{Rem}

\noindent{\bf Acknowledgement}

\noindent
The author would like to thank Prof.\ O. Fujino and the referee for informing him of the relevant references and for kind advice to improve the presentation.

\bigskip
Faculty of Education, Gifu Shotoku Gakuen University

Yanaizu-cho, Gifu City, Gifu 501-6194, Japan

fukuda@ha.shotoku.ac.jp

\end{document}